\documentclass[10pt]{amsart}
\usepackage{graphicx}
\usepackage{amscd}
\usepackage{amsmath}
\usepackage{amsthm}
\usepackage{amsfonts}
\usepackage{amssymb}
\usepackage{mathrsfs}
\usepackage{enumerate}
\usepackage{amsrefs}
\usepackage[pagebackref,hypertexnames=false, colorlinks, citecolor=blue, linkcolor=red, pdfstartview=FitB]{hyperref}
\usepackage[latin1]{inputenc}

\newcommand{\D}{\operatorname{\mathbb{D}}}

\newcommand{\J}{\operatorname{\mathfrak{J}}}
\newcommand{\Z}{\operatorname{\mathbb{Z}}}

\newcommand{\C}{\operatorname{\mathbb{C}}}

\newcommand{\B}{\operatorname{\mathcal{B}}}
\newcommand{\A}{\operatorname{\mathcal{A}}}

\newcommand{\hil}{\operatorname{\mathcal{H}}}
\newcommand{\kil}{\operatorname{\mathcal{K}}}

\newcommand{\e}{\operatorname{\varepsilon}}

\newcommand{\ol}{\overline }
\let\phi\varphi

\newtheorem{lemma}{Lemma}[section]
\newtheorem{theorem}[lemma]{Theorem}
\newtheorem{proposition}[lemma]{Proposition}
\newtheorem{corollary}[lemma]{Corollary}
\theoremstyle{definition}

\begin{document}
\author{Rapha\"el Clou\^atre}
\address{Department of Mathematics, Indiana University, 831 East 3rd Street,
Bloomington, IN 47405} \email{rclouatr@indiana.edu}
\title[Unitary equivalence and similarity to Jordan models]{Unitary equivalence and similarity to Jordan models for weak contractions of class $C_0$}
%\date{\today}
\subjclass[2010]{Primary:47A45, 47L55}
%\keywords{$C_0$ operators, invariant subspaces, quasisimilarity}
\begin{abstract}
We obtain results on the unitary equivalence of weak contractions of class $C_0$ to their Jordan models under an assumption on their commutants. In particular, our work addresses the case of arbitrary finite multiplicity. The main tool is the
theory of boundary representations due to Arveson. We also
generalize and improve previously known results concerning unitary equivalence and similarity to Jordan models when the minimal function is a Blaschke product.
\end{abstract}
\maketitle

\section{Introduction}

We start with some background concerning operators of class $C_0$
(greater detail can be found in \cite{bercovici1988} or
\cite{nagy2010}). Let $H^\infty$ be the algebra of bounded
holomorphic functions on the open unit disc $\D$. Let $\hil$ be a
Hilbert space and $T$ a bounded linear operator on $\hil$, which we
indicate by $T\in B(\hil)$. If $T\in B(\hil)$ is a completely
non-unitary contraction, then its associated Sz.-Nagy--Foias
$H^\infty$ functional calculus is an algebra homomorphism $\Phi:
H^\infty \to B(\hil)$ with the following properties:
\begin{enumerate}[(i)]
    \item $\|\Phi(u)\|\leq u$ for every $u\in H^\infty$
    \item $\Phi(p)=p(T)$ for every polynomial $p$
    \item $\Phi$ is continuous when $H^\infty$ and $B(\hil)$ are equipped with their respective weak-star topologies.
\end{enumerate}
We use the notation $\Phi(u)=u(T)$ for $u\in H^\infty$. The
contraction $T$ is said to belong to the class $C_0$ whenever $\Phi$ has a
non-trivial kernel. It is known in that case that $\ker \Phi=\theta
H^\infty$ for some inner function $\theta$ called the minimal
function of $T$, which is uniquely determined up to a scalar factor
of absolute value one.

We denote by $H^2$ the Hilbert space of functions
$$f(z)=\sum_{n=0}^\infty a_n z^n$$
holomorphic on the open unit disc, equipped with the norm
$$
\|f\|_{H^2}=\left(\sum_{n=0}^\infty |a_n|^2\right)^{1/2}.
$$
For any inner function $\theta\in H^\infty$, the space
$H(\theta)=H^2\ominus \theta H^2$ is closed and invariant for $S^*$,
the adjoint of the shift operator $S$ on $H^2$. The operator
$S(\theta)$ defined by $S(\theta)^*=S^*|(H^2\ominus \theta H^2)$ is
called a Jordan block; it is of class $C_0$ with minimal function
$\theta$.

A more general family of operators consists of the so-called
Jordan operators. Start with a collection of inner functions $\Theta=\{\theta_{\alpha}\}_{\alpha}$ indexed by the ordinal numbers such that $\theta_\alpha=1$ for $\alpha$ large enough and that $\theta_{\beta}$ divides $\theta_{\alpha}$ whenever $\text{card}(\beta)\geq \text{card}(\alpha)$ (recall that a
function $u\in H^\infty$ divides another function $v\in H^\infty$ if
$v=uf$ for some $f\in H^\infty$). Let $\gamma$ be the first ordinal such that $\theta_\gamma=1$. Then, the associated Jordan operator is $J_\Theta=\bigoplus_{\alpha<\gamma} S(\theta_{\alpha})$.

The Jordan
operators are of fundamental importance in the study of operators of
class $C_0$ as the following theorem from \cite{bercovici1975}
illustrates. Recall first that a bounded injective linear operator
with dense range is called a quasiaffinity. Two operators
$T\in B(\hil)$ and $T'\in B(\hil')$ are said to be
quasisimilar if there exist quasiaffinities $X:\hil\to
\hil'$ and $Y:\hil' \to \hil$ such that $XT=T' X$ and $T Y=YT'$.

\begin{theorem}\label{t-classification}
For any operator $T$ of class $C_0$ there exists a unique Jordan operator
$J$ which is quasisimilar to $T$.
\end{theorem}

This theorem is one of the main features of the class
$C_0$. Recent investigations have identified special situations in
which the relation of quasisimilarity between a multiplicity-free
operator $T$ of class $C_0$ and its Jordan model can be improved to
similarity. For instance, the work done in \cite{clouatre2011} was
inspired in part by the early results of Apostol found in
\cite{apostol1976} (discovered independently in
\cite{lwilliams1976}). A link was found between the possibility
of achieving similarity between $T$ and $S(\theta)$ and the fact that
$\phi(T)$ has closed range for every inner divisor $\phi$ of
$\theta$ (here $\theta$ denotes the minimal function of $T$). The
same problem was studied in \cite{clouatre2013}, albeit from another
point of view. Drawing inspiration from the seminal work of Arveson
\cite{arveson1969}, the main question addressed in that paper was
whether similarity between $T$ and $S(\theta)$ could be detected via
properties of the associated algebras
$$
H^{\infty}(T)=\{ u(T):u\in H^\infty\}
$$
and
$H^{\infty}(S(\theta))$. More precisely, assuming that these
algebras are boundedly isomorphic, does it follow that $T$ and
$S(\theta)$ are similar? Partial results along with estimates on the
size of the similarity were obtained in \cite{clouatre2013} in the
case where the minimal function is a finite Blaschke product. In both
\cite{clouatre2011} and \cite{clouatre2013} the
considerations also took advantage of (and perhaps reinforced) a
well-known connection with the theory of interpolation by bounded
holomorphic functions on the unit disc and the so-called
(generalized) Carleson condition (see \cite{vasyunin1976} or
\cite{nikolskii1979}).

Our work here offers several improvements and generalizations of
various results from \cite{arveson1969},\cite{arveson1972},
\cite{clouatre2011} and \cite{clouatre2013}. As mentioned above, our
focus is to describe a relation between an operator
of class $C_0$ and its Jordan model, and we do so in
two different settings: up to similarity and up to unitary
equivalence. We now present the plan of the paper, state our main
results and explain to what extent those improve upon previous ones.

Section 2 is based on the following result, which is a consequence
of the proof of Theorem 3.6.12 in \cite{arveson1969} and of
Corollary 1 in \cite{arveson1972}, both due to Arveson. Recall that
a vector $x\in \hil$ is said to be cyclic for $T\in B(\hil)$ if the
smallest closed subspace of $\hil$ containing $T^n x$ for every
integer $n\geq 0$ is the entire space $\hil$. An operator having a
cyclic vector is said to be multiplicity-free. We denote by $P(T)$
the smallest norm-closed algebra containing $T$ and the identity
operator.

\begin{theorem}\label{t-bdryarveson}
Let $T\in B(\hil)$ be an irreducible multiplicity-free operator of
class $C_0$ with minimal function $\theta$ and with the
property that its spectrum does not contain the unit circle. Consider the homomorphism
$$
\Psi:P(S(\theta))\to P(T)
$$
defined by $\Psi(p(S(\theta)))=p(T)$ for every polynomial $p$. Assume that $\Psi$ is completely isometric. Then, $T$ is
unitarily equivalent to $S(\theta)$.
\end{theorem}

Our first main result stated below addresses the case of
higher multiplicities and removes the condition on the spectrum of
$T$. We denote by $\{T\}'$ the commutant of the operator $T$.

\begin{theorem}\label{t-intromain1}
Let $T_1\in B(\hil_1)$ be an operator of class $C_0$ with the
property that $I-T_1^*T_1$ is of trace class and that $\{T_1\}'$ is
irreducible. Let $T_2\in B(\hil_2)$ be another operator of class
$C_0$ which is quasisimilar to $T_1$  and with the property that
$\{T_2\}'$ is irreducible. Assume that there exists a completely
isometric isomorphism $$\phi:\{T_1\}'\to \{T_2\}'$$ such that
$\phi(T_1)=T_2$. Then, $T_1$ and $T_2$ are unitarily equivalent.
\end{theorem}

In Section 3, we explore the case where the minimal function is a Blaschke product and show that in this setting, unitary equivalence between a multiplicity-free operator of class $C_0$ and its Jordan model can be obtained from assumptions weaker than those appearing in the statement above. Throughout, we use
the notation
$$
b_{\lambda}(z)=\frac{z-\lambda}{1-\overline{\lambda}z}
$$
for the Blaschke factor with root at $\lambda\in \D$ and
$$
\tilde{b}_{\lambda}(z)=-\frac{\ol{\lambda}}{|\lambda|}b_{\lambda}(z)=-\frac{\ol{\lambda}}{|\lambda|}\frac{z-\lambda}{1-\overline{\lambda}z}.
$$
Given a Blaschke product $\theta$, an inner divisor $\psi$ of
$\theta$ is said to be \textit{big} if the ratio $\theta/\psi$
is a Blaschke factor. Also, a multiplicity-free contraction of class $C_0$ whose
minimal function is a Blaschke product is said to be
\textit{maximal} if there exists a big divisor $\psi$ of $\theta$
and a unit cyclic vector $\xi$ with the property that
$\|\psi(T)\xi\|=1$. The motivation for Section 3 is the following result due to
Arveson (see Lemma 3.2.6 in \cite{arveson1969}).

\begin{theorem}\label{t-maxarveson}
Let $T\in B(\hil)$ be a multiplicity-free operator of class $C_0$
whose minimal function is a finite Blaschke product $\theta$. Assume that $T$ is maximal.
Then, $T$ is unitarily equivalent to $S(\theta)$.
\end{theorem}

The following is our second main result. One improvement that it offers over the previous theorem is the possibility for $\theta$ to be an infinite Blaschke product.

\begin{theorem}\label{t-intromain2}
Let $T\in B(\hil)$ be a multiplicity-free operator of class $C_0$
whose minimal function is a Blaschke product
$\theta$. Assume that
$\|\psi(T)\|=1$ for some big inner divisor $\psi$ of $\theta$. Then, $T$ is unitarily equivalent
to $S(\theta)$.
\end{theorem}

Let us also emphasize here that the condition $\|\psi(T)\|=1$ is formally weaker than maximality: although cyclic vectors are known to be plentiful (see Theorem \ref{t-maxvector}), it is not immediately clear that an operator must achieve its norm on one of them. That this is indeed the case follows from the proof of Theorem \ref{t-maxarveson} and to the best of our knowledge it was not observed before.

Finally, in Section 4 we are concerned with similarity rather than unitary equivalence. The basic idea is to weaken the
condition appearing in Theorem \ref{t-intromain2} while still obtaining
similarity between $T$ and its Jordan model. Our results  improve upon the work that was done in \cite{clouatre2013}.

\textit{Acknowledgements}: The author would like to thank Hari Bercovici for several fruitful discussions, and for pointing the existence of the paper \cite{moore1974}.

\section{Unitary equivalence and boundary representations}
In this section, we investigate $*$-representations of
$C^*$-algebras related to $C_0$ operators and their connection with
unitary equivalence of such operators to their Jordan models. The
first result we need is inspired by the discussion found  on page 201 of
\cite{arveson1969}.  We denote by $\kil(\hil)$ the ideal
of compact operators on a Hilbert space $\hil$.
\begin{lemma}\label{l-repCT}
Let $T\in B(\hil)$ be an operator which is not unitary and with the property that $I-TT^*$ and $I-T^*T$ are compact  and
that $\{T\}'$ is irreducible.
\begin{enumerate}
\item[\rm{(i)}]  If we denote by $\J$ the closed ideal  of $C^*(\{T\}')$ generated by $I-T^*T$ and $I-TT^*$, then $\J=\kil(\hil)$.

\item[\rm{(ii)}] Assume that $\pi$ is a
$*$-representation of $C^*(\{T\}')$. Then, $\pi(T)$ is unitarily
equivalent to $ (T\otimes I_{\hil'}) \oplus U, $ where $U$ is a
unitary operator with spectrum contained in the essential spectrum
of $T$ and $\hil'$ is another Hilbert space.
\end{enumerate}
\end{lemma}
\begin{proof}
The fact that $C^*(\{T\}')$ is irreducible immediately implies that
the ideal $\J$ is irreducible (see Lemma I.9.15  in
\cite{davidson1996}). By assumption, $\J$ is a non-zero $C^*$-subalgebra of
$\kil(\hil)$. Thus $\kil(\hil)=\J$ by Corollary I.10.4 of
\cite{davidson1996}, which proves (i). Moreover, Lemma 3.4.4 of
\cite{arveson1969} shows that the representation $\pi$ can be
decomposed as
$$
\pi(x)=\pi_1(x)\oplus\pi_2(x+\J)
$$
for every $x\in C^*(\{T\}')$, where the $*$-representation $\pi_1$
is the unique extension to $C^*(\{T\}')$ of a $*$-representation of
$\J$ and $\pi_2$ is a $*$-representation of $C^*(\{T\}')/\J$. Since
$\J=\kil(\hil)$, it is well-known (see Corollary I.10.7 of
\cite{davidson1996}) that $\pi_1|\J$ must be unitarily equivalent to a
multiple of the identity representation, and by uniqueness so must be $\pi_1$. On the hand, $\pi_2(T)$ is
a unitary operator with spectrum contained in the essential spectrum
of $T$, which shows (ii) and finishes the proof.
\end{proof}

We note that in the statement above we allow for both $\hil'$ and
the space on which $U$ acts to be zero. In other words, one of the
pieces $U$ or $T\otimes I_{\hil'}$ can be absent. This is the case
if we specialize Lemma \ref{l-repCT} to contractions of
class $C_0$.

\begin{lemma}\label{l-C0rep}
Let $T_1\in B(\hil_1)$ be an operator of class $C_0$ with the
property that $\{T_1\}'$ is irreducible and $I-T_1^*T_1, I-T_1T_1^*$
are compact. Let $T_2\in B(\hil_2)$ be another operator of class
$C_0$. Assume that there exists a $*$-homomorphism
$$
\pi:C^*(\{T_1\}')\to B(\hil_2)
$$
such that $\pi(T_1)=T_2$. Then,
$T_2$ is unitarily equivalent to $T_1\otimes I_{\hil'}$ for some
Hilbert space $\hil'$.
\end{lemma}
\begin{proof}
By virtue of Lemma \ref{l-repCT}, we see that $T_2=\pi(T_1)$ is
unitarily equivalent to $(T_1\otimes I_{\hil'})\oplus U$ for some
unitary $U$. Since $T_2$ is of class $C_0$, it must be completely
non-unitary and thus $U$ acts on the zero space, so that $T_2$ is in
fact unitarily equivalent to $ T_1\otimes I_{\hil'}.$
\end{proof}

Next, we need some results of Bercovici and Voiculescu (see
\cite{bercovici1977}). Recall that a contraction $T$ is said to be \textit{weak} if $I-T^*T$ belongs to the ideal of trace class operators.
\begin{theorem}\label{t-compactdefect}
Let $T$ be an operator of class $C_0$ with Jordan model $J$. Then,
\begin{enumerate}
\item[\rm{(i)}] $T$ is a weak contraction if and only if $T^*$ is a weak contraction
\item[\rm{(ii)}] $T$ is a weak contraction if and only if $J$ is a weak contraction
%\item[\rm{(iii)}] $I-T^*T$ is compact if $I-J^*J$ is compact.
\end{enumerate}
\end{theorem}

\begin{lemma}\label{l-weakdirectsum}
Let $T_1$ and $T_2$ be quasisimilar weak contractions of class
$C_0$. If $T_1$ is unitarily equivalent to $T_2\otimes I_{\hil'}$,
then $\hil'$ is one dimensional and $T_1$ is unitarily equivalent to
$T_2$.
\end{lemma}
\begin{proof}
This follows immediately from a consideration of the determinant
functions of $T_1$ and $T_2$, which must be equal (see section 6.3
of \cite{bercovici1988} for more details).
\end{proof}

The next corollary is the link between $*$-representations and unitary equivalence.

\begin{corollary}\label{c-C0rep}
Let $T_1\in B(\hil_1)$ be a weak contraction of class $C_0$ with the
property that $\{T_1\}'$ is irreducible. Let $T_2\in B(\hil_2)$ be another operator of class $C_0$
which is quasisimilar to $T_1$.  Assume that there exists a
$*$-homomorphism 
$$
\pi:C^*(\{T_1\}')\to B(\hil_2)
$$
such that
$\pi(T_1)=T_2$. Then, $T_2$ is unitarily equivalent to $T_1$.
\end{corollary}
\begin{proof}
Since $T_2$ and $T_1$ are quasisimilar, they share the same Jordan
model. By Theorem \ref{t-compactdefect}, we see that $I-T^*_iT_i$
and $I-T_i T_i^*$ are of trace class for $i=1,2$. In light of
Lemma \ref{l-C0rep}, we know that $T_2$ is unitarily equivalent to
$T_1\otimes I_{\hil'}$, and an application of Lemma
\ref{l-weakdirectsum} shows that $T_1$ and $T_2$ are unitarily
equivalent.
\end{proof}

It is typically quite difficult to construct $*$ -representations of
$C^*(\{T\}')$ in order to apply this result. We follow here one method
to obtain such representations which is originally due to Arveson
\cite{arveson1969}. Recall that given a unital (non-self-adjoint) subalgebra $\A \subset B(\hil_1)$, we say that an irreducible
$*$-representation $$\pi:C^*(\A)\to B(\hil)$$ is a \textit{boundary
representation} for $\A$ if the only unital completely positive
extension of $\pi|\A$ to $C^*(\A)$ is $\pi$ itself (we refer the reader to \cite{arveson1969} or \cite{paulsen2002} for further details and definitions).
Our next goal is to establish that for weak contractions of class
$C_0$ with irreducible commutant, the identity representation of
$C^*(\{T\}')$ is a boundary representation for $\{T\}'$. The main
tool is the following result, known as Arveson's boundary theorem
(see Theorem 2.1.1 in \cite{arveson1972}).

\begin{theorem}\label{t-bdryrep}
Let $\A\subset B(\hil)$ be an irreducible unital subalgebra
with the property that $C^*(\A)$ contains $\kil(\hil)$ and that the
quotient map $$q:B(\hil)\to B(\hil)/\kil(\hil)$$ is not completely
isometric on $\A$. Then, the identity representation of $C^*(\A)$ is
a boundary representation for $\A$.
\end{theorem}

In order to apply this theorem, we also require the following fact from \cite{moore1974}.

\begin{theorem}\label{t-compcomm}
If $T\in B(\hil)$ is an operator of class $C_0$ with the property
that $I-T^*T$ and $I-TT^*$ are compact, then there exists a function
$u\in H^\infty$ with the property that $u(T)$ is a non-zero compact operator.
\end{theorem}

We now achieve the desired result. In fact, we only require $I-T^*T$ and $I-TT^*$ to be compact, as opposed to trace class.

\begin{corollary}\label{c-idbdry}
If $T\in B(\hil)$ is an operator of class $C_0$ such that $I-T^*T$
and $I-TT^*$ are compact and $\{T\}'$ is irreducible, then the
identity representation of $C^*(\{T\}')$ is a boundary
representation for $\{T\}'$.
\end{corollary}
\begin{proof}
First, we see that $C^*(\{T\}')$ contains $\kil(\hil)$ by virtue of
Lemma \ref{l-repCT} (i). Moreover, by Theorem \ref{t-compcomm} there
exists a non-zero compact operator of the form $u(T)$ for some $u\in
H^\infty$. This operator necessarily commutes with $T$ and
$q(u(T))=0$, and thus Theorem \ref{t-bdryrep} completes the proof.
\end{proof}

The following result is Theorem 1.2 of \cite{arveson1969}. It is the
key to obtaining $*$-representa\-tions of $C^*(\{ T\}')$.

\begin{theorem}\label{t-pullbackisomet}
Let $\A\subset B(\hil_1), \B\subset B(\hil_2)$ be unital subalgebras and let $$\phi:\A\to \B$$ be a unital completely isometric
algebra isomorphism. Let $\pi_{\B}$ be a $*$-representa\-- tion of
$C^*(\B)$ which is a boundary
representation for $\B$. Then, there exists a $*$-representation
$\pi_{\A}$ of $C^*(\A )$ which is
a boundary representation for $\A$ and such that
$\pi_{\B}\circ\phi=\pi_{\A}$ on $\A$.
\end{theorem}

Finally, we come to the main result of this section.

\begin{theorem}\label{t-bdryue}
Let $T_1\in B(\hil_1)$ be a weak contraction of class $C_0$ with the
property that $\{T_1\}'$ is irreducible. Let $T_2\in B(\hil_2)$ be another operator of class $C_0$
which is quasisimilar to $T_1$ and with the property that $\{T_2\}'$ is irreducible. Assume that there
exists a completely isometric isomorphism 
$$
\phi:\{T_1\}'\to \{T_2\}'
$$
such that $\phi(T_1)=T_2$. Then, $T_1$ and $T_2$ are
unitarily equivalent.
\end{theorem}
\begin{proof}
By Theorem \ref{t-compactdefect}, we see that $I-T_i^*T_i$ and $I-T_iT_i^*$ are of trace class for each $i=1,2$. In light of Corollary \ref{c-idbdry}, the identity representation
of $C^*(\{T_2\}')$ is a boundary representation for $\{T_2\}'$.
Therefore, we may apply Theorem \ref{t-pullbackisomet} to obtain a
$*$-representation $\pi: C^*(\{T_1\}')\to B(\hil_2)$ which
satisfies $\pi|\{T_1\}'=\phi$. An application of Corollary
\ref{c-C0rep} finishes the proof.
\end{proof}

We close this section by examining more closely the irreducibility
assumption appearing above. Obviously, the main interest of Theorem
\ref{t-bdryue} lies in the case where $T_2$ is the Jordan model of
$T_1$. In that case, the irreducibility assumption on $\{T_1\}'$ is
necessary to obtain unitary equivalence in view of the following
fact.

\begin{proposition}\label{p-Jcommirred}
If $J=\bigoplus_{\alpha}S(\theta_{\alpha})$ is a Jordan operator, then $\{J\}'$ is irreducible.
\end{proposition}
\begin{proof}
Set $\hil=\bigoplus_{\alpha}H(\theta_{\alpha})$. Let $M\subset \hil$ be a proper reducing
subspace for $\{J\}'$. Let
$P_{\alpha}$ denote the orthogonal projection of $\hil$ onto the $H(\theta_{\alpha})$ component. Since $P_{\alpha}$ commutes with $J$ for every $\alpha$, we see that $M$ is reducing for each $P_{\alpha}$ and hence it can be written as
$M=\bigoplus_{\alpha}M_{\alpha}$. Now, the operator
$P_{\alpha}JP_{\alpha}$ also commutes with $J$, whence each $M_{\alpha}$ is reducing for $P_{\alpha}JP_{\alpha}$. Since this operator is unitarily equivalent to $S(\theta_{\alpha})$ and Jordan blocks are known to be irreducible, we must have either
$M_{\alpha}=P_{\alpha}\hil$ or $M_{\alpha}=0$. We proceed to
show that each $M_{\alpha}$ must be equal to $0$. For the rest of the proof, for each $\alpha$ we identify $P_{\alpha}\hil$ with $H(\theta_{\alpha})$.

Since we assume that $M\neq \hil$, we must have $M_{\alpha_0}=0$ for
some $\alpha_0$. Now, any operator $X$ acting on $\hil$ may be written as
$X=(X_{\alpha \beta})_{\alpha,\beta}$, where 
$$
X_{\alpha
\beta}:H(\theta_{\beta})\to H(\theta_{\alpha}).
$$
If
$\gamma<\alpha_0$, consider the operator $Y(\gamma)$ defined by
$Y(\gamma)_{\alpha_0
\gamma}=P_{H(\theta_{\alpha_0})}|H(\theta_{\gamma})$ and
$Y(\gamma)_{\alpha \beta}=0$ otherwise. It is easily verified that
$Y(\gamma)$ commutes with $J$ and thus  $Y(\gamma)M\subset M$,
which in turn implies that
$$
P_{H(\theta_{\alpha_0})}M_{\gamma}\subset M_{\alpha_0}=0.
$$
This forces $M_{\gamma}$ to be equal to $0$, since the other possibility $M_{\gamma}=H(\theta_{\gamma})$ is impossible:
$$
P_{H(\theta_{\alpha_0})}H(\theta_{\gamma})
=H(\theta_{\alpha_0})\neq 0.
$$
Therefore,  $M_{\gamma}=0$ whenever $\gamma<\alpha_0$. Assume now that $\gamma>\alpha_0$ and
consider the operator $Z(\gamma)$ defined by
$$
Z(\gamma)_{\alpha_0 \gamma}=P_{H(\theta_{\alpha_0})}\left(\theta_{\alpha_0}/
\theta_{\gamma}\right)(S)|H(\theta_{\gamma})
$$
and $Z(\gamma)_{\alpha \beta}=0$ otherwise. It is easily verified that $Z(\gamma)$ commutes with $J$ and thus  $Z(\gamma)M\subset M$, which in turn implies that
$$
P_{H(\theta_{\alpha_0})}(\theta_{\alpha_0}/
\theta_{\gamma})(S)M_{\gamma}\subset
M_{\alpha_0}=0.
$$
This forces $M_{\gamma}$ to be equal to $0$, since the other possibility $M_{\gamma}=H(\theta_{\gamma})$ is impossible:
$$
P_{H(\theta_{\alpha_0})}(\theta_{\alpha_0}/
\theta_{\gamma})(S)H(\theta_{\gamma})
=(\theta_{\alpha_0}/
\theta_{\gamma})(S)H(\theta_{\gamma})\neq 0.
$$
Thus, $M_{\gamma}=0$ for every $\gamma>\alpha_0$, and the proof is complete.
\end{proof}

We obtain a simpler version of Theorem \ref{t-bdryue} that applies to Jordan operators.

\begin{corollary}\label{c-bdryue}
Let $T\in B(\hil)$ be a weak contraction of class $C_0$ with the
property that $\{T\}'$ is irreducible. Let $J$ be the Jordan model of $T$. Assume that there
exists a completely isometric isomorphism 
$$
\phi:\{T\}'\to\{J\}'
$$
such that $\phi(T)=J$. Then, $T$ and $J$ are
unitarily equivalent.
\end{corollary}
\begin{proof}
Simply combine Theorem \ref{t-bdryue} and Proposition \ref{p-Jcommirred}.
\end{proof}

Finally, we show that for certain minimal functions, the
irreducibility of $\{T\}'$ is automatic. We first need two
preliminary facts. The first one is from \cite{bercovici1975}.

\begin{theorem}\label{t-doublecomm}
Let $T$ be an operator of class $C_0$ with minimal function $\theta$
and let $X\in \{T\}''$. Then, there exists a function $v\in H^\infty$
with the property that $v$ has no non-constant common inner divisor
with $\theta$ and that $Xv(T)=u(T)$ for some function $u\in
H^\infty$.
\end{theorem}

The following is Proposition 2.4.9 of \cite{bercovici1988}.

\begin{lemma}\label{l-qa}
Let $u\in H^\infty$ and $T\in B(\hil)$ be an operator of class $C_0$
with minimal function $\theta$. Then, $u(T)$ is a quasiaffinity if
and only if $u$ and $\theta$ have no non-constant common inner
divisor.
\end{lemma}

We now show that if the inner divisors of the minimal function
$\theta$ satisfy a certain property, then the commutant $\{T\}'$ is
always irreducible. We recall that the double commutant of $T$, denoted by $\{T\}''$, is defined as the algebra of operators that commute with the commutant:
$$
\{T\}''=(\{T\}')'=\{A\in B(\hil): AX=XA \text{ for every} X\in \{T\}'\}.
$$

\begin{proposition}\label{p-Tcommirred}
Let $\theta\in H^\infty$ be an inner function with the property that
for every inner divisor $\phi$ of $\theta$, we have that $\phi$ and
$\theta/\phi$ have a non-constant common inner divisor, unless
$\phi=1$ or $\phi=\theta$. Let $T$ be an operator of class $C_0$
with minimal function $\theta$. Then, $\{T\}''$ contains no
idempotents besides $0$ and $I$, and $\{T\}'$ is irreducible.
\end{proposition}
\begin{proof}
The second statement clearly follows from the first, so it suffices
to show that if $E\in \{T\}''$ satisfies $E^2=E$, then $E=I$ or
$E=0$. By Theorem \ref{t-doublecomm}, we see that there $Ev(T)=u(T)$
for some functions $u,v\in H^\infty$ where $v$ and $\theta$ have no non-constant
common inner divisor. We compute
\begin{align*}
u(T)^2&= E^2 v(T)^2\\
&=E v(T)^2\\
&=u(T)v(T)
\end{align*}
whence
$$
u^2-uv=\theta f
$$
for some $f\in H^\infty$. If we define $\phi\in H^\infty$ to be the
greatest common inner divisor of $u$ and $\theta$, we can write $
u=\phi g $ where $g$ and $\theta$ have no non-constant common inner divisor. Now,
we see that
$$
\phi g^2-gv=\frac{\theta}{\phi}f,
$$
which implies that the greatest common inner divisor of $\phi$ and
$\theta/\phi$ divides $gv$. By choice of $g$ and $v$ and by
assumption on $\theta$, we see that $\phi=1$ or $\phi=\theta$. If
$\phi=\theta$, then $u(T)=0$ and the equation
$$
0=u(T)=Ev(T)
$$
along with Lemma \ref{l-qa} implies that $E=0$. If $\phi=1$, then
$u(T)$ is a quasiaffinity by Lemma \ref{l-qa} which forces $E$ to be
a quasiaffinity as well, by virtue of the equation
$$
Ev(T)=v(T)E=u(T).
$$
But $E$ has closed range (being idempotent), and thus it must be invertible. The equation $E^2=E$ then yields $E=I$, and the proof is complete.
\end{proof}
A moment's thought reveals that an inner function $\theta$
satisfying the condition of the previous proposition must be of one of two
types: either a power of a Blaschke factor
$\theta(z)=(b_{\lambda}(z))^n$ or a singular inner function
associated to a point mass measure on the unit circle
$$
\theta(z)=\exp\left(t\frac{z+\zeta}{z-\zeta} \right)
$$
where $t>0$ and $|\zeta|=1$. In fact, these are the inner functions
whose inner divisors are completely ordered by divisibility (see Proposition 4.2.6 in \cite{bercovici1988}).
Moreover, Proposition \ref{p-Tcommirred} extends a recent result of Jiang and Yang (see \cite{yang2011}) which deals with the case of $T$ being a Jordan block $S(\theta)$. In this special case, the result holds under the weaker condition that the function $\theta$ does not admit a so-called corona decomposition. We now formulate another corollary of the main result of this section.

\begin{corollary}\label{c-singue}
Let $T_1\in B(\hil_1)$ be a weak contraction of class $C_0$ with minimal function $\theta$. Assume that the inner divisors of $\theta$ are completely ordered by divisibility.
Let $T_2\in B(\hil_2)$ be another operator of class $C_0$
which is quasisimilar to $T_1$. Assume that there
exists a completely isometric isomorphism 
$$
\phi:\{T_1\}'\to \{T_2\}'
$$
such that $\phi(T_1)=T_2$. Then, $T_1$ and $T_2$ are
unitarily equivalent.
\end{corollary}
\begin{proof}
This is a mere restatement of Theorem \ref{t-bdryue} using
Proposition \ref{p-Tcommirred} and the discussion that follows its
proof.
\end{proof}

\section{Unitary equivalence and maximality for Blaschke products}
In this section, we show that in the case where the minimal function $\theta$ of a multiplicity-free operator is a Blaschke product, we may replace the assumption on the existence of a completely isometric isomorphism appearing in the previous section by a condition on the norm of a single operator. In fact, we investigate the maximality condition appearing in
Theorem \ref{t-maxarveson} and we set out to prove Theorem
\ref{t-intromain2} which improves it significantly. The first step
is an estimate which, although completely elementary, is very useful
(the reader might want to compare it with Lemma 3.7 of
\cite{clouatre2013}).

\begin{lemma}\label{l-mobius}
Let $T\in B(\hil)$ be a contraction and let $h\in \hil$ such that $\|Th\|\geq \delta \|h\|$ for some $\delta >0$. Then,
$$
\|b_{\mu}(T)h\|\geq \frac{\delta-|\mu|}{1+|\mu|}\|h\|
$$
for every $\mu\in \D$.
\end{lemma}
\begin{proof}
We have that $ b_{\mu}(T)(1-\ol{\mu}T)=T-\mu $ so that
$$
\|b_{\mu}(T)(1-\ol{\mu}T)h\|\geq (\delta-|\mu|)\|h\|.
$$
Thus,
$$
\|b_{\mu}(T)h\|\geq
\frac{1}{\|1-\ol{\mu}T\|}\|b_{\mu}(T)(1-\ol{\mu}T)h\|\geq
\frac{\delta-|\mu|}{1+|\mu|}\|h\|.
$$
\end{proof}

We also require the following fact (see \cite{nagy1970comp}).

\begin{theorem}\label{t-invsub}
Let $T\in B(\hil)$ be a multiplicity-free operator of class $C_0$
with minimal function $\theta$. Then, every closed invariant
subspace $M\subset \hil$ of $T$ is of the form
$$
M=\ker \phi(T)=\ol{(\theta/\phi)(T)\hil}
$$ for
some inner divisor $\phi$ of $\theta$. Conversely, if $\phi$ is an inner divisor of $\theta$, then
$$
M=\ker \phi(T)=\ol{(\theta/\phi)(T)\hil}
$$
is an invariant subspace for $T$ and the minimal function of $T|M$ is equal to $\phi$.
\end{theorem}

The next lemma is used to prove the main result of this section,
but it is also of independent interest. The very basic Lemma \ref{l-mobius} first comes into play here.

\begin{lemma}\label{l-cyclic}
Let $T\in B(\hil)$ be a multiplicity-free operator of class $C_0$ whose minimal function is a Blaschke product $\theta$. Let
$\xi\in \hil$ be a unit vector satisfying $\|\psi(T)\xi\|=1$ for
some big inner divisor $\psi$ of $\theta$. Then, $\xi$ is cyclic.
\end{lemma}
\begin{proof}
Let $M\subset \hil$ be the smallest closed
invariant subspace for $T$ which contains $\xi$. By Theorem
\ref{t-invsub}, there must exist an inner divisor $\phi$ of $\theta$ with
the property that $M=\ker \phi(T)$. The desired conclusion will follow if we show that $\phi=\theta$, for then $M=\hil$. Assume
on the contrary that $\phi$ is a proper divisor of $\theta$. Then,
there exists a big divisor $\omega$ of $\theta$ with the property
that $\omega(T)\xi=0$. Note that $\psi(T)\xi\neq 0$ by assumption so that
$\psi\neq \omega$. Now, there exists $\lambda,\mu\in \D$ distinct zeros of $\theta$ such that $\psi=\theta/b_{\lambda}$ and $\omega=\theta/b_{\mu}$.
Choose $z\in \D$ with the property that $b_{z}
\circ b_{\mu}=b_{\lambda}$. Using Lemma \ref{l-mobius} and the fact
that
$$
1=\|\psi(T)\xi\|=\left\|b_{\mu}(T)\left(\frac{\theta}{b_{\lambda}b_{\mu}}\right)(T)\xi\right\|,
$$
we find
\begin{align*}
0&=\|\omega(T)\xi\|\\
&=\left\|b_{\lambda}(T) \left(\frac{\theta}{b_{\lambda}b_{\mu}}\right)(T)\xi\right\|\\
&=\left\|(b_{z} \circ b_{\mu})(T)\left(\frac{\theta}{b_{\lambda}b_{\mu}}\right)(T)\xi\right\|\\
&\geq \frac{1-|z|}{1+|z|}\left\|\left(\frac{\theta}{b_{\lambda}b_{\mu}}\right)(T)\xi\right\|\\
&\geq \frac{1-|z|}{1+|z|}\|\psi(T)\xi\|
\end{align*}
which is a contradiction since $\psi(T)\xi\neq 0$.
\end{proof}

Before moving on to the next step towards the main result of this section, we recall an elementary fact (see \cite{clouatre2013} for instance).
\begin{lemma}\label{l-basisfinite}
Let $T\in B(\hil)$ be a multiplicity-free operator of class $C_0$
with minimal function $\theta=b_{\lambda_1}\ldots b_{\lambda_N}$.
Let $\xi\in \hil$ be a cyclic vector for $T$.
Then, the vectors
$$
\xi,b_{\lambda_1}(T)\xi,(b_{\lambda_1}b_{\lambda_2})(T)\xi,\ldots,(b_{\lambda_1}\ldots b_{\lambda_{N-1}})(T)\xi
$$
form a basis for
$\hil$.
\end{lemma}

We require an infinite-dimensional version of Lemma \ref{l-basisfinite}.
First, we need another basic fact (see Theorem 2.4.6 of \cite{bercovici1988})
%
%
%Moreover, given $E\subset \hil$, we
%denote by $\bigvee E$ the smallest closed subspace containing $E$.

\begin{lemma}\label{l-lcmspan}
Let $T\in B(\hil)$ be an operator of class $C_0$ with minimal
function $\theta$. Given a family $\{\theta_n\}_n$ of inner divisors
of $\theta$ with least common inner multiple $\phi$, we have that $\ker\phi(T)$ is the smallest closed subspace containing  $\ker \theta_n(T)$ for every $n$.
%and
%$$
%\ker \psi(T)=\bigcap_{n}\ker \theta_n(T)
%$$
%where $\phi$ (respectively $\psi$) is the least common inner
%multiple (respectively the greatest common inner divisor) of the
%family $\{\theta_n\}_n$.
\end{lemma}

We now proceed to establish a more general version of Lemma \ref{l-basisfinite}.

\begin{lemma}\label{l-basis}
Let $T\in B(\hil)$ be a multiplicity-free operator of class $C_0$
whose minimal function is a Blaschke product
$\theta$. Let $\xi\in \hil$ be a
unit vector which is also cyclic for $T$. Then, for every big divisor $\psi$ of $\theta$ we have that $\hil$ is the
smallest closed subspace containing $\phi(T)\xi$ for every inner
divisor $\phi$ of $\psi$.
\end{lemma}
\begin{proof}
We can write $\theta=\prod_{n=0}^\infty \theta_n$ where each
$\theta_n$ is a power of a Blaschke factor with the property that
$\theta_n$ and $\theta_m$ have no non-constant common inner divisor if $n\neq m$.
For each $n\geq 0$, put
$
\theta_n=\tilde{b}_{\lambda_n}^{d_n}
$
for some $\lambda_n\in \D$ and some positive integer $d_n$.
Without loss of generality, we may assume that $\psi=\theta/b_{\lambda_0}$.

By Lemma \ref{l-basisfinite}, we see that for each $n\geq 1$ the set
$$\{(b_{\lambda_n}^{k})(T)h: 0\leq k \leq d_n\}$$ is a basis for $\ker
(b_{\lambda_0}\theta_n)(T)$ whenever $h$ is a cyclic vector for
$T|\ker(b_{\lambda_0}\theta_n)(T)$. Since $\xi$ is cyclic for $T$,
we get from Theorem \ref{t-invsub} that
$
\left(\theta/(b_{\lambda_0}\theta_n) \right)(T)\xi
$
is cyclic for $T|\ker(b_{\lambda_0}\theta_n)(T)$ and thus
$$
\left\{\left( \frac{\theta}{b_{\lambda_0}\theta_n}
b_{\lambda_n}^{k}\right)(T)\xi: 0\leq k \leq d_n\right\}
$$
is a basis for $\ker (b_{\lambda_0}\theta_n)(T)$ for every $n\geq 1$. Note in addition that
$$
\frac{\theta}{b_{\lambda_0}\theta_n} b_{\lambda_n}^{k}
$$
divides
$\psi=\theta/b_{\lambda_0}$ for every $k$.
A similar argument shows that
$$
\left\{\left( \frac{\theta}{\theta_0} b_{\lambda_0}^{k}\right)(T)\xi:
0\leq k \leq d_0-1\right\}
$$
is a basis for $\ker \theta_0(T)$. Note once again that
$$
\frac{\theta}{\theta_0} b_{\lambda_0}^{k}
$$ divides
$\psi=\theta/b_{\lambda_0}$ for every $0\leq k \leq d_0-1$.

Therefore the smallest closed subspace containing all the vectors
of the form $\phi(T)\xi$ where $\phi$ is an inner divisor of
$\psi$ contains
$$
\ker \theta_0(T)\cup \bigcup_{n=1}^\infty \ker
(b_{\lambda_0}\theta_n)(T).
$$
Since the least common inner multiple of $\theta_0$ and the family
$\{b_{\lambda_0}\theta_n \}_n$ is $\theta$, the conclusion follows from Lemma \ref{l-lcmspan}.
\end{proof}

We now come to the main result of this section. In light of Lemma \ref{l-basis} and Lemma \ref{l-cyclic}, it is a consequence of Theorem 3.2.9 in \cite{arveson1969} and the surrounding circle of ideas. However, we feel that the following proof (which is an adaptation of that of Lemma 3.2.6 in \cite{arveson1969}) is instructive and more direct, so we provide it nonetheless.

\begin{theorem}\label{t-maxue}
Let $T\in B(\hil)$ be a multiplicity-free operator of class $C_0$
whose minimal function is a Blaschke product
$\theta$. Assume that
$\|\psi(T)\|=1$ for some big inner divisor $\psi$ of $\theta$. Then, $T$ is unitarily equivalent
to $S(\theta)$.
\end{theorem}
\begin{proof}
Set $\psi=\theta/b_{\lambda}$. By Theorem \ref{t-invsub}, we know that $N=\ol{\psi(T)\hil}$ is an invariant subspace for $T$ with the property that $T|N$ has minimal function equal to $b_{\lambda}$. Thus, $T|N$ must be quasisimilar to $S(b_{\lambda})$ by Theorem \ref{t-classification} and we conclude that $N$ is one-dimensional. In other words, $\psi(T)$ has rank $1$ and there exists a unit vector $\xi\in
\hil$ with the property that $\|\psi(T)\xi\|=1$. It is easily verified that this implies that $\|\phi(T)\xi\|=1$ for every inner divisor $\phi$ of $\psi$. Note also that the vector $\xi$ is cyclic for $T$ by Lemma \ref{l-cyclic}.

Let us
now denote by $U:\kil\to\kil$ the minimal unitary dilation of $T$. The operator $\phi(U)$ is unitary for every inner divisor $\phi$ of $\psi$, whence
$$
\|\phi(U)\xi\|=1=\|\phi(T)\xi\|.
$$
These equalities coupled with the relation
$$
\phi(T)=P_{\hil}\phi(U)|\hil
$$
force $\phi(T)\xi=\phi(U)\xi$ for every inner divisor $\phi$ of $\psi$. Consequently, for integers $n,m$ such
that $n> m$, we have
\begin{align*}
\left\langle b_{\lambda}^n(U)\xi,b_{\lambda}^m(U)\xi
\right\rangle
&=
\left\langle b_{\lambda}^{n-m}(U)\xi, \xi\right\rangle\\
&=\left\langle \psi(U) b_{\lambda}^{n-m}(U)\xi, \psi(U)\xi\right\rangle\\
&=\left\langle \psi(U) b_{\lambda}^{n-m}(U)\xi,
\psi(T)\xi\right\rangle\\
&=\left\langle (\psi b_{\lambda}^{n-m})(T)\xi,
\psi(T)\xi\right\rangle\\
&=0
\end{align*}
whence $\{b_{\lambda}^n(U)\xi\}_{n\in\Z}$ is an orthonormal set
which generates a Hilbert space $\kil_0\subset \kil$. Define now an
operator $\Lambda: \kil_0\to L^2$ such that
$$
\Lambda b_{\lambda}^n(U)\xi=b_{\lambda}^n k_{\lambda}
$$
for every integer $n$, where $k_{\lambda}\in H^\infty$ is defined as
$$
k_{\lambda}(z)=\frac{1-|\lambda|^2}{1-\ol{\lambda}z}.
$$
It is easily verified that $\Lambda$ is unitary and $\Lambda
b_{\lambda}(U)=b_{\lambda}(M_z) \Lambda$, where $M_z$ denotes
the unitary operator of multiplication by $z$ on $L^2$. Since Blaschke factors can be uniformly approximated on $\ol{\D}$ by
polynomials and
$$
b_{-\lambda}\circ b_{\lambda}(z)=z=b_{\lambda}\circ b_{-\lambda}(z),
$$
we see that
$$
\Lambda U=M_z \Lambda,
$$
$$
\bigvee_{n=0}^\infty b_{\lambda}^n(U)\xi=
\bigvee_{n=0}^\infty U^n \xi
$$
and
$$
\bigvee_{n=0}^\infty b_{\lambda}^nk_{\lambda}=
\bigvee_{n=0}^\infty z^n k_{\lambda}.
$$
If we put $\hil_0=\bigvee_{n=0}^\infty b_{\lambda}^n(U)\xi$, then clearly $\hil_0$ is invariant under $U$ and
\begin{align*}
\Lambda \hil_0
=\bigvee_{n=0}^\infty z^n k_{\lambda}
=H^2
\end{align*}
since
a straightforward calculation shows that $k_{\lambda}$ is an outer
function.
Moreover, we find
\begin{align*}
\Lambda \phi(T)\xi &=\Lambda
\phi(U)\xi\\
&=\phi(M_z)k_{\lambda}\\
&=\phi k_{\lambda}\\
&=\phi(S(\theta)) k_{\lambda}
\end{align*}
for every inner divisor $\phi$ of $\psi$,
and thus $\Lambda \hil=H(\theta)$ by Lemma \ref{l-basis} (here we
used the well-known fact that $k_{\lambda}$ is cyclic for
$S(\theta)$). In particular, we see that
\begin{align*}
\hil =\Lambda^* H(\theta)\subset \Lambda^* H^2= \hil_0.
\end{align*}
If we set $W=\Lambda|\hil$, then we obtain another unitary
operator $W:\hil\to H(\theta)$. Using $\Lambda \hil=H(\theta)$ along
with $\Lambda (\hil_0\ominus \hil)=\theta H^2$, we conclude that
$$
P_{H(\theta)}\Lambda|\hil_0=\Lambda P_{\hil}|\hil_0.
$$
Since $U\hil_0\subset\hil_0$, we have $U\hil\subset \hil_0$ and
\begin{align*}
WT&=(\Lambda|\hil)T\\
&=\Lambda P_{\hil}U| \hil\\
&=(\Lambda P_{\hil}|\hil_0)U| \hil\\
&=(P_{H(\theta)}\Lambda|\hil_0) U|\hil\\
&=P_{H(\theta)}\Lambda U|\hil\\
&=P_{H(\theta)}M_z \Lambda|\hil\\
&=(P_{H(\theta)}M_z|H(\theta))\Lambda|\hil\\
&=S(\theta)(\Lambda|\hil)\\
&=S(\theta)W
\end{align*}
so that $T$ is unitarily equivalent to $S(\theta)$.
\end{proof}

Recall now the following well-known consequence of the commutant lifting
theorem (see \cite{sarason1967}).

\begin{theorem}\label{t-jbcomm}
The map
$$
u+\theta H^\infty \mapsto u(S(\theta))
$$
establishes an isometric algebra isomorphism between
$H^\infty/\theta H^\infty$ and $\{S(\theta)\}'$. In particular,
$$
\|u(S(\theta))\|=\inf\{\|u+\theta
f\|_{H^\infty}: f\in H^\infty \}
$$
for every $u\in H^\infty$.
\end{theorem}

We close this section by stating a simpler version of Theorem
\ref{t-maxue}.

\begin{corollary}\label{c-ciue}
Let $T\in B(\hil)$ be a multiplicity-free operator of class $C_0$
whose minimal function is a Blaschke product
$\theta$. Assume that the map
$$
\Psi:H^\infty(T)\to H^\infty (S(\theta))
$$
defined by $\Psi(u(T))=u(S(\theta))$ is contractive. Then, $T$ is
unitarily equivalent to $S(\theta)$.
\end{corollary}
\begin{proof}
This follows directly from Theorem \ref{t-maxue}. Indeed, if $\lambda$ is a zero of $\theta$, then
\begin{align*}
\left\|\left(\frac{\theta}{b_{\lambda}}\right)(T)\right\| &\geq \left\|\Psi\left(\left(\frac{\theta}{b_{\lambda}}\right)(T)\right)\right\| \\
&= 
\left\|\left(\frac{\theta}{b_{\lambda}}\right)(S(\theta))\right\|\\
&=\inf\left\{\left\|\frac{\theta}{b_{\lambda}}+\theta
f\right\|:f\in H^\infty
\right\}\\
&=\inf\left\{\left\|1+b_{\lambda} f\right\|:f\in H^\infty \right\}\\
&=1
\end{align*}
where we used Theorem \ref{t-jbcomm}.
\end{proof}
Note that in the setting of that corollary, we do not need to assume the irreducibility of the commutant (compare with Theorem \ref{c-bdryue}).

\section{Similarity and lower bounds for big divisors of finite Blaschke products}

The focus of this section
shifts from unitary equivalence to similarity. Let $T$ be a multiplicity-free operator of class $C_0$ whose minimal function $\theta$ is a Blaschke product. We saw in Section 3 (Theorem \ref{t-maxue}) that under the assumption that $\|\psi(T)\|=1$ for some big divisor $\psi$ of $\theta$, then $T$ and $S(\theta)$ must be unitarily equivalent. In this final section, we investigate the possibility of obtaining a weaker conclusion, namely similarity, from a weaker assumption on the norm of $\psi(T)$. This problem was studied in \cite{clouatre2013} where the following partial result was obtained.

\begin{theorem}\label{t-simalgeta}
Let $T_1\in B(\hil_1), T_2\in B(\hil_2)$ be multiplicity-free
operators of class $C_0$ with minimal function
$\theta=b_{\lambda_1}\ldots b_{\lambda_N}$. Define
$$
\eta=\sup_{1\leq j,k\leq
N}\frac{|b_{\lambda_j}(\lambda_k)|^{1/2}}{(1-\max
\{|\lambda_j|,|\lambda_k| \}^2)^{1/2}}.
$$
Assume that
$$
\|(\theta/b_{\lambda_N})(T_1)\|>  \beta+5\sqrt{2}\eta.
$$
and
$$
\|(\theta/b_{\lambda_N})(T_2)\|> \beta+5\sqrt{2}\eta
$$
for some constant $\beta$ satisfying
$$
\left(1-\frac{1}{(N-1)^2}\right)^{1/2}<\beta<1.
$$
Then, there exists an invertible operator $X:\hil_1\to \hil_2$ such
that $XT_1=T_2 X$ and
$$
\max\{\|X\|,\|X^{-1}\| \}\leq C(\beta,N)
$$
where $ C(\beta,N)>0$ is a constant depending only on
$\beta$ and $N$.
\end{theorem}
We should remark at this point that in the above setting the spaces $\hil_1$ and $\hil_2$ are finite dimensional, and thus Theorem \ref{t-classification} implies that $T_1$ and $T_2$ must be similar. Thus, the relevance of Theorem \ref{t-simalgeta} lies in the control over the norm of the similarity rather than in its existence. This control allows one to obtain similarity results for infinite Blaschke products having certain nice properties. We refer the curious reader to \cite{clouatre2013} for such applications  related to interpolation by bounded holomorphic
functions on the unit disc.

On the other hand, the presence of the
quantity $\eta$ in the previous statement is unexpected and seems
artificial. Moreover, it has the unpleasant consequence of
restricting the minimal functions to which the theorem applies since clearly $\eta$ must be smaller than $(5\sqrt{2})^{-1}$. The
main result of this section removes $\eta$ completely at the cost of a
slightly stronger assumption on the operators (which is
automatically satisfied by Jordan blocks, however). In particular, it applies
to arbitrary finite Blaschke products.

The main technical tool we require is the following fact which can be inferred from the work done in
\cite{clouatre2013}.

\begin{theorem}\label{t-simalg} Let $T_1\in B(\hil_1), T_2\in
B(\hil_2)$ be multiplicity-free operators of class $C_0$ whose
minimal function is a finite Blaschke product $\theta$ with $N$ roots. Assume
that there exist unit cyclic vectors $\xi_1\in \hil_1,\xi_2\in
\hil_2$ and a constant $\beta$ with the property that
$$
\|\phi(T_1)\xi_1\|\geq \beta>\left(1-\frac{1}{(N-1)^2}\right)^{1/2}
$$
and
$$
\|\phi(T_2)\xi_2\|\geq \beta> \left(1-\frac{1}{(N-1)^2}\right)^{1/2}
$$
for every inner divisor $\phi$ of $\theta$. Then, there exists an
invertible operator $X:\hil_1\to \hil_2$ such that $XT_1=T_2 X$ and
$$
\max\{\|X\|,\|X^{-1}\| \}\leq C(\beta,N)
$$
where $ C(\beta,N)>0$ is a constant depending only on
$\beta$ and $N$.
\end{theorem}

Before we proceed, we establish some auxiliary results.
The first one is well-known, but we provide the proof for the reader's convenience.

\begin{lemma}\label{l-coronasim}
Let $\theta_1,\theta_2\in H^\infty$ be inner functions such that
there exist $u_1,u_2\in H^\infty$ with the property that $\theta_1
u_1+\theta_2 u_2=1$. Let $T\in B(\hil)$ be an operator of class
$C_0$ with minimal function $\theta_1 \theta_2$. Then, there exists
an invertible operator $X$ such that
$$
XTX^{-1}= T|\ker \theta_1(T)\oplus T|\ker \theta_2(T),
$$
$$
\|X\|\leq \left(\|u_1\|^2_{H^\infty}+\|u_2\|^2_{H^\infty} \right)^{1/2}
$$
and
$$
\|X^{-1}\|\leq \sqrt{2}.
$$
\end{lemma}
\begin{proof}
Define
$$
X:\hil \to \ker \theta_1(T)\oplus \ker \theta_2(T)
$$
as
$$
Xf=(\theta_2 u_2)(T)f\oplus(\theta_1 u_1)(T) f
$$
for every $f\in \hil$. If $Xf=0$, then
$$
f=(\theta_1 u_1+\theta_2 u_2)(T)f=0
$$
and thus $X$ is injective. Given $g_1\in \ker \theta_1(T)$ and
$g_2\in \ker \theta_2(T)$, we see that
\begin{align*}
Xg_1&=(\theta_2 u_2)(T)g_1\oplus 0\\
&=(1-\theta_1u_1)(T)g_1\oplus 0\\
&=g_1\oplus 0
\end{align*}
and
\begin{align*}
Xg_2&=0\oplus (\theta_1 u_1)(T)g_2\\
&=0\oplus (1-\theta_2 u_2)(T)g_2\\
&=0\oplus g_2
\end{align*}
which shows that $X$ is surjective. Notice also that
$$
XT=(T\oplus T)X.
$$
Therefore, we see that $T$ is similar to
$$
T|\ker \theta_1(T)\oplus T|\ker \theta_2(T).
$$
It remains only to estimate the norm of $X$ and $X^{-1}$. For
$f\in \hil$, we have
\begin{align*}
\|Xf\|&=\|(\theta_2 u_2)(T)f\oplus(\theta_1
u_1)(T) f\|\\
&=\left(\| (\theta_1 u_1)(T)f\|^2+\| (\theta_2 u_2)(T)f\|^2 \right)^{1/2} \\
&\leq  \left(\| (\theta_1 u_1)(T)\|^2+\| (\theta_2 u_2)(T)\|^2 \right)^{1/2}\|f\|\\
&\leq \left(\| u_1\|_{H^\infty}^2+\| u_2\|_{H^\infty}^2 \right)^{1/2}\|f\|
\end{align*}
and
\begin{align*}
\|f\|&=\|(\theta_1 u_1+\theta_2 u_2)(T)f\|\\
&\leq \|(\theta_1 u_1)(T)f\|+\|
(\theta_2 u_2)(T)f\|\\
&\leq \sqrt{2}(\|(\theta_1 u_1)(T)f\|^2+\|
(\theta_2 u_2)(T)f\|^2)^{1/2}\\
&= \sqrt{2} \|Xf\|
\end{align*}
by the Cauchy-Schwarz inequality. This completes the proof.
\end{proof}

\begin{lemma}\label{l-coronasimest}
Let $\theta_1,\theta_2\in H^\infty$ be inner functions such that
$$
\inf_{z\in \D} \{|\theta_1(z)|+|\theta_2(z)|\}=\delta>0.
$$Let $T\in B(\hil)$ be an operator of class
$C_0$ with minimal function $\theta_1 \theta_2$. Then, there exists
an invertible operator $X$ such that
$$
XTX^{-1}= T|\ker \theta_1(T)\oplus T|\ker \theta_2(T),
$$
$\|X^{-1}\|\leq \sqrt{2}$ and $\|X\|\leq C(\delta)$, where
$C(\delta)>0$ is a constant depending only on $\delta$.
\end{lemma}
\begin{proof}
This is an easy consequence of Lemma \ref{l-coronasim} and the
estimates associated to Carleson's corona theorem (see
\cite{carleson1962} or Theorem 3.2.10 of \cite{nikolskii2002}).
\end{proof}

In applying that lemma, the following estimate will prove to be useful.

\begin{lemma}\label{l-lowerbound}
Let $E,F\subset \D$ be two finite subsets of cardinality at most
$N$, and let $\theta_E,\theta_F\in H^\infty$ be the associated
Blaschke products. Assume that there exists $r>0$ such that $|e-f|\geq r$ for every
$e\in E,f\in F$. Then,
$$
\inf_{z\in \D}\{|\theta_E(z)|+|\theta_F(z)|
\}>(r/4)^N.
$$
\end{lemma}
\begin{proof}
Throughout the proof we put
$$
d(A,B)=\inf\{|a-b|:a\in A,b\in B\}
$$
whenever $A,B\subset \C$. 

First note that
$$
|b_{\lambda}(z)|=\left|\frac{z-\lambda}{1-\ol{\lambda}z} \right|\geq
\frac{|z-\lambda|}{2}
$$
for every $z\in \D$. In particular, we see that
$$
|\theta_E(z)|\geq (r/4)^N
$$
for every $z\in \D$ such that $d(z,E)\geq r/2$. Now, if $d(z,E)< r/2$ then $d(z,F)\geq r/2$ in view of the triangle
inequality and of our assumption on the sets $E$ and $F$. Thus, we
conclude that
$$
|\theta_F(z)|\geq (r/4)^N
$$
if $d(z,E)<r/2$. Combining these inequalities yields
$$
|\theta_E(z)|+|\theta_F(z)|\geq (r/4)^N
$$
for every $z\in \D$.
\end{proof}

Next, we need an elementary combinatorial lemma.

\begin{lemma}\label{l-comb}
Let $\e>0$ and $\lambda_1,\ldots,\lambda_N\in \C$. Then, there exists an integer $1\leq k \leq N$
with the property that the set $\{\lambda_1,\ldots,\lambda_N\}$ can
be written as the disjoint union of
$$
E_k=\{\lambda_j: |\lambda_j-\lambda_1|<\e 2^{-(N+1-k)}\}
$$
and
$$
F_k=\{\lambda_j: |\lambda_j-\mu|\geq \e 2^{-(N+1-k)} \text{ for
every } \mu \in E_k\}.
$$
\end{lemma}
\begin{proof}
Put $S_N=\{\lambda_1,\ldots,\lambda_N\}$. It is clear that $E_k$ and
$F_k$ are disjoint and that $\lambda_1\in E_k$ for every $1\leq k \leq N$. Consider the set
$G_k=S_N\setminus (E_k\cup F_k)$ for every $1\leq k\leq N$. An
element $\lambda_j$ lies in $G_k$ if it does not belong to $E_k$ but
there exists $\mu \in E_k$ with the property that
$$
|\lambda_j-\mu|<\frac{\e}{2^{N+1-k}}.
$$
By the triangle inequality, we see that $G_k \cup E_k\subset
E_{k+1}$. If $G_k$ is non-empty for each $1\leq k\leq N-1$, this last inclusion implies that $E_{k}$ contains at least $k$ elements for each $1\leq k \leq N$, so that
$E_N=S_N$ and $G_N$ is empty, and the lemma follows.
\end{proof}

One last bit of preparation is necessary. The next fact is found in \cite{nagy1971}  (it was independently
discovered by Herrero, see \cite{herrero1972}).

\begin{theorem}\label{t-maxvector}
Let $T\in B(\hil)$ be a multiplicity-free operator of class $C_0$.
Then, the set of cyclic vectors for $T$ is a dense $G_\delta$ in
$\hil$.
\end{theorem}

Finally, we come to our similarity result which improves Theorem
\ref{t-simalgeta} in the sense that it removes any restriction on
the roots of the minimal function $\theta$.

\begin{theorem}\label{t-simfiniteBP}
Let $T_1\in B(\hil_1), T_2\in B(\hil_2)$ be multiplicity-free
operators of class $C_0$ whose minimal function is a finite Blaschke product $\theta$ with $N$ roots. Assume that there exist constants $\beta',\beta$ such that
$$
\|\phi(T_1)|\ker \psi(T_1)\|>\beta'>\beta> \left(1-\frac{1}{(N-1)^2}
\right)^{1/4}
$$
and
$$
\|\phi(T_2)|\ker \psi(T_2)\|>\beta'>\beta> \left(1-\frac{1}{(N-1)^2}
\right)^{1/4}
$$
whenever $\psi$ is a non-constant inner divisor of $\theta$ and $\phi$ is a proper inner divisor of $\psi$. Then, there exists an invertible operator $X$ with the property that $XT_1=T_2X$
and
$$
\max\{ \|X\|, \|X^{-1}\|\}\leq C(N,\beta,\beta'),
$$
where $C(N,\beta,\beta')>0$ is a
constant depending only on $N$, $\beta$ and $\beta'$.
\end{theorem}
\begin{proof}
Put $\theta=b_{\lambda_1}\ldots b_{\lambda_N}$.
We proceed by induction on $N$. The case $N=1$ is trivial since then the equations
$$
b_{\lambda_1}(T_1)=b_{\lambda_1}(T_2)=0
$$
imply that $T_1$ and $T_2$ are equal to the same multiple of the identity operator. Assume that the conclusion holds for
Blaschke products with at most $N-1$ roots. For each $1\leq k \leq
N$ we set $\psi_k=\theta/b_{\lambda_k}$. Since $
\|\psi_N(T_i)\|>\beta', $ by Theorem \ref{t-maxvector} we can find a
unit cyclic vector $\xi_i \in \hil$ such that $
\|\psi_N(T_i)\xi_i\|> \beta' $ for $i=1,2$. For $1\leq k <N$ we see
that
$$
\psi_k(T_i)\xi_i=b_{\lambda_N}(T_i)\left(\frac{\theta}{b_{\lambda_k}b_{\lambda_N}}\right)(T_i)\xi_i
$$
while
$$
\psi_N(T_i)\xi_i=b_{\lambda_k}(T_i)\left(\frac{\theta}{b_{\lambda_k}b_{\lambda_N}}\right)(T_i)\xi_i
$$
and thus by Lemma \ref{l-mobius} we find
\begin{align*}
\|\psi_k(T_i)\xi_i\|&=\left\|b_{\lambda_N}(T_i)\left(\frac{\theta}{b_{\lambda_k}b_{\lambda_N}}\right)(T_i)\xi_i \right\|\\
&= \left\|(b_{\mu}\circ b_{\lambda_k})(T_i)\left(\frac{\theta}{b_{\lambda_k}b_{\lambda_N}}\right)(T_i)\xi_i \right\| \\
&\geq \frac{ \beta'-|\mu|}{1+|\mu|}\left\|\left(\frac{\theta}{b_{\lambda_k}
b_{\lambda_N}}\right)(T_i)\xi_i\right\|\\
&\geq \frac{ \beta'-|\mu|}{1+|\mu|}\|\psi_N(T_i)\xi_i\|\\
&\geq \left(\frac{ \beta'-|\mu|}{1+|\mu|}\right)\beta'
\end{align*}
where $\mu=-b_{\lambda_N}(\lambda_k)$. Choose now $r>0$ such that
$$
\left(\frac{ \beta'-|\mu|}{1+|\mu|}\right)\beta'>\beta^2
$$
if $|\lambda_k-\lambda_N|<r$. Clearly, $r$ depends only on $\beta$ and $\beta'$ and if $ |\lambda_k-\lambda_N|<r, $ then $
\|\psi_k(T_i)\xi_i\|> \beta^2. $
Thus, the desired conclusion follows from
Theorem \ref{t-simalg} in case where $\sup_{1\leq k \leq
N}|\lambda_k-\lambda_N|<r$. Assume therefore that this supremum is at
least $r$. In that case, Lemma \ref{l-comb} allows us to write
$$
\{\lambda_1,\ldots,\lambda_N\}=E \cup F
$$
where $E$ and $F$ are disjoint and non-empty,
$|\lambda-\lambda_N|<r$ for every $\lambda\in E$, and
$|\lambda-\mu|>r2^{-N}$ for every $\lambda \in E, \mu\in F$. Let $\theta_E$ (respectively $\theta_F$) be the
Blaschke product associated to the elements of $E$ (respectively
$F$). By Lemmas \ref{l-coronasimest} and \ref{l-lowerbound}, for each $i=1,2$
there exists an invertible operator $Y_i$ with the property that
$$Y_iTY_i^{-1}=T_i|\ker\theta_E(T_i)\oplus T_i|\ker \theta_F(T_i)$$ and 
$$
\max\{\|Y_i\|,
\|Y_i^{-1}\|\}\leq C_1
$$
where $C_1>0$ depends only on $N$, $\beta$ and $\beta'$. Note now that the minimal function of $T_i|\ker \theta_E(T_i)$ (respectively $T_i|\ker \theta_F(T_i)$) is $\theta_E$ (respectively $\theta_F$) by virtue of Theorem \ref{t-invsub}. Since $E$ and
$F$ have cardinality strictly less than $N$, we are done by
induction.
\end{proof}

In case where one of the operators is a Jordan block, we obtain a
simpler version of the previous result by making use of another
property of Jordan blocks, found in Proposition 3.1.10 of
\cite{bercovici1988}.

\begin{lemma}\label{l-unitequiv}
Let $\phi$ be an inner divisor of the inner function $\theta$. Then,
the operator $S(\theta)|\ker \phi(S(\theta))$ is unitarily
equivalent to $S(\phi)$.
\end{lemma}

\begin{corollary}\label{c-simfiniteBP}
Let $T\in B(\hil)$ be a multiplicity-free operator of class $C_0$ whose minimal function is a finite Blaschke product $\theta$ with at most $N$ roots. Assume that there exist constants $\beta,\beta'$ such that
$$
\|\phi(T)|\ker \psi(T)\|> \beta'>\beta> \left(1-\frac{1}{(N-1)^2}
\right)^{1/4}
$$
whenever $\psi$ is an inner divisor of $\theta$ and $\phi$ is an inner divisor of $\psi$. Then, there exists an invertible operator $X$ with the property that $XT=S(\theta)X$
and
$$
\max\{ \|X\|, \|X^{-1}\|\}\leq C(N,\beta,\beta'),
$$
where $C(N,\beta,\beta')>0$ is a
constant depending only on $N$,$\beta$ and $\beta'$.
\end{corollary}
\begin{proof}
This follows directly from Lemma \ref{l-unitequiv}, Theorem \ref{t-jbcomm} and Theorem \ref{t-simfiniteBP}.
\end{proof}

As was done in \cite{clouatre2013}, this corollary can be applied to obtain similarity results for some infinite Blaschke products. Pursuing those applications here would lead us outside of the intented scope of the paper, so let us simply mention that the proofs follow the same lines as those from \cite{clouatre2013}.

%\bibliography{biblio}
\bibliography{/home/raphael/Dropbox/Research/biblio}
\bibliographystyle{plain}

\end{document}